\documentclass[12pt]{article}

\title{$L^q$-spectra of measures on planar non-conformal attractors}

\usepackage{multicol}
\usepackage{float}
\usepackage{amsmath}
\usepackage{amsthm}
\usepackage{amsmath}
\usepackage{amsthm}
\usepackage{amscd}
\usepackage{amsfonts}
\usepackage{amssymb}
\usepackage{latexsym}
\usepackage{verbatim}
\usepackage{enumerate}
\usepackage{graphicx}
\usepackage{empheq}
\usepackage[utf8]{inputenc}
\usepackage[T1]{fontenc}
\usepackage{lmodern} 
\usepackage{parskip}
\usepackage{url}
\usepackage{fancyhdr}
\usepackage{xcolor}

\newtheorem{thm}{Theorem}[section]
\newtheorem{cor}[thm]{Corollary}
\newtheorem{lem}[thm]{Lemma}

\theoremstyle{definition}
\newtheorem{defn}[thm]{Definition}

\numberwithin{equation}{section}

\allowdisplaybreaks

\date{}

\author{Kenneth J. Falconer, Jonathan M. Fraser \& Lawrence D. Lee }

\begin{document}
\maketitle
\begin{abstract} 
We study the $L^q$-spectrum of measures in the plane generated by certain nonlinear  maps. In particular we consider attractors of iterated function systems consisting of maps whose components are $C^{1+\alpha}$ and for which the Jacobian is a lower triangular matrix at every point subject to a natural domination condition on the entries. We calculate the $L^q$-spectrum of Bernoulli measures supported on such sets by using an appropriately defined analogue of the singular value function and an appropriate pressure function.


\end{abstract}

\section{Introduction}

The study of fractals generated by iterated function systems (IFSs) consisting of nonlinear maps, which can often be identified with repellers of corresponding dynamical systems, has a rich history. In 1994 Falconer  \cite{FalconerBD} calculated the dimension of mixing repellers for non-conformal mappings. To do this he applied techniques from thermodynamic formalism, in particular developing a subadditive version of the theory and also a ``bounded distortion'' principle. Further work on nonlinear IFSs was done by Hu who in 1996 calculated the box and Hausdorff dimensions of invariant sets of expanding $C^2$ maps \cite{Hu}. More recently Cao, Pesin and Zhao \cite{Pesin} studied the Hausdorff dimension of non-conformal repellers corresponding to $C^{1+\alpha}$ maps. By studying certain subadditive and superadditive pressures   they were able to obtain bounds for the Hausdorff dimension of repellers.

Other notable work in this area was done in 2007 by Manning and Simon  \cite{Manning} who investigated the subadditive pressure of nonlinear maps developed by Falconer applied to nonlinear maps and considered cases where bounded distortion does not hold. The work of Falconer as well as that of Manning and Simon and also Miao \cite{Miao} was generalised by Barany \cite{Barany} who used the subadditive pressure to calculate the Hausdorff dimension of fractals generated by IFSs whose maps have triangular Jacobians. Other authors to have considered IFSs generated by triangular mappings include Kolossv\'ary and Simon \cite{Istvan}. In particular they looked at a family of planar self-affine carpets with overlaps generated by lower triangular matrices and considered whether dimension drop occurs.

In terms of multifractal analysis Falconer studied the $L^q$-spectrum of self-affine measures \cite{Falconerlq1} and almost self-affine measures \cite{Falconerlq2}. In the case of self-affine measures he was able to establish a generic formula in the region $1<q\leq 2$ in terms of a subadditive pressure expession. Barral and Feng \cite{Barral} then generalised this in certain cases to calculate the $L^q$-spectrum for a wider range of $q$ and were also able to verify the multifractal formalism in some cases. For results on the $L^q$-spectrum of  measures on self-affine carpets, see Feng and Wang  \cite{Feng} and Fraser \cite{Fraser}.

In this paper we calculate the $L^q$-spectra of Bernoulli measures in the plane supported on sets generated by IFSs consisting of $C^{1+\alpha}$ maps whose Jacobian matrices are lower triangular. Our approach is based on setting up certain `almost-additive' pressure functionals. As a corollary we calculate the box dimension of the supports of these measures. Our results on $L^q$-dimensions are new, even in the (non-diagonal) self-affine case.

Standard background on iterated function systems may be found, for example, in \cite{Falconer, Hut}. We introduce further definitions, in particular \textit{nonlinear attractors} and \textit{nonlinear measures} which have a particular meaning in this paper as shorthand for the types of non-conformal attractors and measures we consider.
\vspace{5mm}

\begin{defn}[Nonlinear attractor]\label{ifsdef}
Let $\mathcal{I}$ be a finite index set with $|\mathcal{I}|\geq 2$ and let $\lbrace S_{i}\rbrace_{i\in\mathcal{I}}$ be an IFS consisting of contractions on $[0, 1]^2$. Suppose also that each $S_i:[0, 1]^2\rightarrow [0, 1]^2$ is of the form $S_i(a_1, a_2)= (f_i(a_1), g_i(a_1, a_2))$, where the $f_i$ and $g_i$ are $C^{1+\alpha}$ contractions ($0<\alpha \leq 1$) on $[0, 1]$ and $[0, 1]^2$ respectively, that is their derivatives satisfy H\"{o}lder conditions of exponent $\alpha$. (We use one-sided derivatives on the boundary of $[0,1]^2$.) By Hutchinson's theorem \cite{Hut} there is a unique non-empty, compact set $F$ satisfying
\begin{equation*}
F=\bigcup_{i\in\mathcal{I}}S_{i}(F)
\end{equation*}
which for the purposes of this paper we call the \emph{nonlinear attractor} associated to $\lbrace S_{i}\rbrace_{i\in\mathcal{I}}$.
\end{defn}

\vspace{5mm}

We are interested in the natural Bernoulli measures supported on nonlinear attractors $F$, see \cite{Falconer Techniques, Hut}.

\vspace{5mm}

\begin{defn}[Nonlinear measure]\label{probs}
Let $F$ be a nonlinear attractor given by  $\lbrace S_{i}\rbrace_{i\in\mathcal{I}}$ on $[0, 1]^2$, and  let $\lbrace p_{i}\rbrace_{i\in\mathcal{I}}$ be a probability vector with each $p_{i}\in (0, 1)$. Then there is a unique Borel probability measure $\mu$ supported on $F$ which satisfies 
\[
\mu=\sum_{i\in \mathcal{I}}p_{i}\ \mu  \circ  S_{i}^{-1}
\]
which we call the \emph{nonlinear measure} associated to $\lbrace S_{i}\rbrace_{i\in\mathcal{I}}$ and  $\lbrace p_{i}\rbrace_{i\in\mathcal{I}}$.
\end{defn}

Our aim is to calculate the \textit{$L^q$-spectra} of these measures.
Let $\delta>0$ and write $\mathcal{D}_{\delta}$ to denote the set of closed cubes in the $\delta$-mesh on $\mathbb{R}^n$ that have positive $\mu$-measure. Write
\begin{equation}\label{momsum}
\mathcal{D}^q_{\delta}(\mu)=\sum_{Q\in\mathcal{D}_{\delta}}\mu(Q)^q.
\end{equation}

\vspace{5mm}

\begin{defn}\label{lq}
If $\mu$ is a compactly supported Borel probability measure on $\mathbb{R}^n$ then for $q\geq 0$ the  \textit{upper} and  \textit{lower} $L^q$-\textit{spectrum} of $\mu$ are defined to be 
\begin{equation}\label{lq1}
\overline{\tau}_{\mu}(q)= \overline{\lim}_{\delta\rightarrow 0}\frac{\log \mathcal{D}^q_{\delta}(\mu)}{-\log\delta}
\end{equation}
and
\begin{equation}\label{lq2}
\underline{\tau}_{\mu}(q)= \underline{\lim}_{\delta\rightarrow 0}\frac{\log \mathcal{D}^q_{\delta}(\mu)}{-\log\delta}
\end{equation}
respectively. If these values coincide then define the $L^q$-\textit{spectrum} of $\mu$, denoted by $\tau_{\mu}(q)$, to be their common value.
\end{defn}

The $L^q$-spectrum can be thought of as an analogue of box-counting dimension for measures; indeed the upper and lower box dimensions of the support of $\mu$ are easily seen to be given by $\overline{\tau}_{\mu}(0)$ and $\underline{\tau}_{\mu}(0)$ respectively. Note that $\tau_{\mu}(1)=0$ (as $\mu$ is a probability measure) and that $\tau_{\mu}$ is decreasing in $q$. Furthermore the $L^q$-spectrum is central in multifractal analysis: in certain key cases  the fine multifractal spectrum of $\mu$ can be obtained by taking the Legendre transform of $\tau_{\mu}$ in which case we say that the multifractal formalism holds (see for instance \cite{Falconer, Olsen}). Another useful property of the $L^q$-spectrum is that if it is differentiable at $1$ then the Hausdorff dimension of the measure $\mu$ is given by $\dim _{H}\mu=-\tau_{\mu}^{\prime}(1)$ \cite{Ngai}.

For  our calculations of $L^q$-spectra for nonlinear measures we require the following separation condition for the IFS.

\vspace{5mm}

\begin{defn}[Rectangular open set condition]\label{ROSC}
An IFS  on $\mathbb{R}^2$ satisfies the \textit{rectangular open set condition} (ROSC) if  $\lbrace S_{i}((0, 1)^2)\rbrace_{i\in\mathcal{I}}$ are pairwise disjoint subsets of the open unit square $(0, 1)^2$.
\end{defn}

Fraser \cite{Fraser} calculated the $L^q$-spectrum $\tau_{\mu}(q)$ of a class of \textit{self-affine} measures in the plane. We broadly follow his approach although there  are several technical challenges which arise due to the nonlinearity, as well as the maps giving rise to non-diagonal Jacobians.

Our main result Theorem \ref{Main Theorem} requires some more assumptions and technical details, in particular that the  $\lbrace S_{i}\rbrace_{i\in\mathcal{I}}$ contract more in the vertical direction than in the horizontal direction. The theorem is stated fully in Section 3 but the essence of it is captured in the following version.

\vspace{5mm}

\begin{thm}\label{Heuristic}
Let $\mu$ be a nonlinear measure which satisfies a natural domination condition and the ROSC and let $q\geq 0$. Then there exists a function $\gamma:[0, \infty)\rightarrow\mathbb{R}$, defined in terms of the probability vector $\lbrace p_{i}\rbrace_{i\in\mathcal{I}}$, the singular values of Jacobian matrices of iterates of the $\lbrace S_{i}\rbrace_{i\in\mathcal{I}}$ and the $L^q$-spectrum of the projection of $\mu$ onto the $x$-axis, such that
\[
\tau_{\mu}(q)=\gamma(q).
\]
\end{thm}

We set up the pressure formalism that enables us to define $\gamma$ in Section 3 and prove the theorem in Section 4.  A simple corollary is that if $\mu$ satisfies the ROSC then the box dimension of the support of $\mu$ is given by $\gamma(0)$.

\section{An Example}

Here we provide an example of a nonlinear IFS and corresponding nonlinear attractor generated by three maps. The maps are
\begin{align*}
&S_1(x, y)=\left(\frac{3x}{5}+\frac{3x^2}{40}, \ \frac{x^2}{12}+\frac{y}{6}\right),\\ 
&S_2(x, y)=\left(\frac{4x}{5}-\frac{4x^3}{30}+\frac{1}{3}, \ \frac{x^2}{10}+\frac{y}{4}+\frac{17}{50}\right),\\
&S_3(x, y)= \left(\frac{3x}{5}, \ \frac{x^2}{10}+\frac{y}{5}+\frac{y^3}{9}+\frac{26}{45}\right).                                  
\end{align*}
\begin{figure}[H]
  \centering
\includegraphics[width=0.8\textwidth]{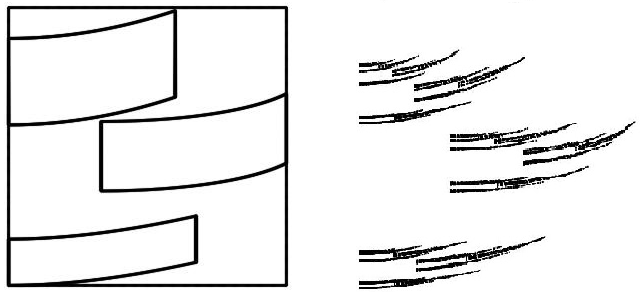}
\caption{The image of the unit square $[0, 1]^2$ under the maps $S_1, S_2$ and $S_3$ is shown on the left; the IFS satisfies the ROSC \eqref{ROSC}. On the right is the corresponding nonlinear attractor. } 
\end{figure}
These maps satisfy the conditions for Theorem \ref{Heuristic} and Theorem \ref{Main Theorem}. The ROSC and domination condition \eqref{Domination} (stated formally in Section 3) are easy to check.  Indeed using Maple software gives 
\begin{align*}
&\inf_{\textbf{a}\in[0, 1]^2}f_{1,x}(\textbf{a})=3/5>\sup_{\textbf{a}\in[0, 1]^2}g_{1,y}(\textbf{a})=1/6\geq\inf_{\textbf{a}\in[0, 1]^2}g_{1,y}(\textbf{a})=1/6\\
&\inf_{\textbf{a}\in[0, 1]^2}f_{2,x}(\textbf{a})=2/5>\sup_{\textbf{a}\in[0, 1]^2}g_{2,y}(\textbf{a})=1/4\geq\inf_{\textbf{a}\in[0, 1]^2}g_{2,y}(\textbf{a})=1/4\\
&\inf_{\textbf{a}\in[0, 1]^2}f_{3,x}(\textbf{a})=3/5>\sup_{\textbf{a}\in[0, 1]^2}g_{3,y}(\textbf{a})=8/15\geq\inf_{\textbf{a}\in[0, 1]^2}g_{3,y}(\textbf{a})=1/5
\end{align*}
with $d=1/6$, say.
Thus any nonlinear measures supported on the attractor of this IFS would fall under the class considered.

\section{A singular value function and pressure}

In \cite{Fraser} Fraser introduced a \textit{q-modified singular value function}. As he was dealing with self-affine measures he needed to consider the singular values of the linear part of each affine map in the IFS. In our nonlinear setting we shall instead consider singular values of Jacobian matrices.

Let $\lbrace S_{i}\rbrace_{i\in\mathcal{I}}$ be an iterated function system of the form in Definition \ref{ifsdef}. 
For $\textbf{a}=(a_1, a_2)\in [0, 1]^2$ and $i\in\mathcal{I}$ we denote the derivative of $S_{i}$ by $D_{\textbf{a}} S_i$. Note that as each $S_i$ is of the form $S_i(a_1, a_2)= (f_i(a_1), g_i(a_1, a_2))$ 
the Jacobian matrix of $D_{\textbf{a}} S_i$ is a lower triangular matrix. 
To simplify notation we will  write $S_i(\textbf{a})= (f_i(\textbf{a}), g_i(\textbf{a}))$, where $\textbf{a} = (a_1, a_2)$, even though $f$ does not depend on $a_2$. If we now write $f_x$ for the derivative of $f$ and $g_x$ and $g_y$ for the partial derivatives of $g$ then
\[
D_{\textbf{a}} S_i = \begin{pmatrix} 
f_{i,x}(\textbf{a}) & 0 \\
g_{i,x}(\textbf{a}) & g_{i,y}(\textbf{a})
\end{pmatrix}.
\] 
From now on we assume that the IFS satisfies the following \textit{domination condition}, which is our key technical assumption.

\vspace{5mm}

\begin{defn}[Domination Condition]
We say the IFS $\lbrace S_{i}\rbrace_{i\in\mathcal{I}}$ satisfies  the \textit{domination condition} if for each map $S_i$ the following inequalities on the derivatives  hold:
\begin{equation}\label{Domination} 
\inf_{{\bf a}\in[0, 1]^2}f_{i,x}({\bf a})>\sup_{{\bf a}\in[0, 1]^2}g_{i,y}({\bf a})\geq \inf_{{\bf a}\in[0, 1]^2}g_{i,y}({\bf a})\geq d,  
\end{equation}
where $d>0$.
\end{defn}
Let 
\begin{equation}\label{Gamma1}
\eta:=\sup_{i\in\mathcal{I}, \textbf{a},\textbf{b}\in[0, 1]^2}\left\{\frac{g_{i,y}(\textbf{a})}{f_{i,x}(\textbf{b})}\right\}<1
\end{equation}
using \eqref{Domination}. In the obvious way we will say that $\mu$ and $F$ satisfy the domination condition if their defining IFS does.

There is no requirement on $g_{i,x}$ to be positive; in particular since this allows $g_{i,x}(\textbf{a})=0$ for all $\textbf{a}$   the class of measures we consider includes self-affine measures supported on Bedford-McMullen carpets, as well as measures supported on attractors of nonlinear ``diagonal'' IFSs.

Let $\mathcal{I}^*=\bigcup_{k\geq 1}\mathcal{I}^k$ denote the set of all finite sequences with entries in $\mathcal{I}$. For $\mathfrak{i}=(i_{1},\dots, i_{k})\in\mathcal{I}^k$ let $S_{\mathfrak{i}} = S_{i_{1}\cdots i_{k}} := S_{i_{1}}\circ S_{i_{2}}\circ \cdots\circ S_{i_{k}}$ and let $p(\mathfrak{i})= p_{i_{1}}p_{i_{2}}\cdots p_{i_{k}}$. 

We write $0<c<1$ for the maximum contraction ratio of the $S_i$ so in particular
\begin{equation}\label{cont}
 |S_{i_{1}\cdots i_{k}}(\textbf{a}) - S_{i_{1}\cdots i_{k}}(\textbf{b})| \leq c^k|\textbf{a} - \textbf{b}| \qquad ((i_{1},\dots, i_{k})\in\mathcal{I}^k,\  \textbf{a}, \textbf{b}\in[0, 1]^2).
\end{equation}

By the chain rule the Jacobian of the composed maps $S_{\mathfrak{i}}$ must be lower triangular, so let $f_{\mathfrak{i},x}(\textbf{a}), g_{\mathfrak{i},x}(\textbf{a}), g_{\mathfrak{i},y}(\textbf{a})$ denote the entries of $D_{\textbf{a}} S_{\mathfrak{i}}$, that is
\begin{equation}\label{Der}
D_{\textbf{a}} S_{\mathfrak{i}} = \begin{pmatrix} 
f_{\mathfrak{i},x}(\textbf{a}) & 0 \\
g_{\mathfrak{i},x}(\textbf{a}) & g_{\mathfrak{i},y}(\textbf{a})
\end{pmatrix}.
\end{equation} 

We will show that the domination condition implies that these matrices satisfy a \textit{bounded distortion}  property which will be key in calculating the $L^q$-spectra.


For $x, y\in\mathbb{R}^+$ we write $x\lesssim y$ to mean that $x\leq C y$ for some absolute constant $C>0$. If we wish to emphasize that this constant depends on some other parameter, $\theta$ say, we write $x\lesssim_{\theta} y$. If both $x\lesssim y$ and $y\lesssim x$ we write $x\asymp y$. In this case we say that $x$ and $y$ are \textit{comparable}.

\vspace{5mm}



Using the chain rule the diagonal entries of \eqref{Der} can be written in terms of derivatives of the individual $f_i$ and $g_i$ as follows.
\begin{equation}\label{fgchain}
f_{\mathfrak{i},x}(\textbf{a})= \prod_{j=1}^k f_{i_j,x}(S_{i_{j+1}\cdots i_k}\textbf{a}), 
\quad g_{\mathfrak{i},y}(\textbf{a})= \prod_{j=1}^k g_{i_j,y}(S_{i_{j+1}\cdots i_k}\textbf{a}).
\end{equation}
(Here and elsewhere make the natural convention that $S_{i_{k+1}}S_{i_{k}}$ is the identity.)
Note that  from \eqref{Gamma1}, using these expansions, 
\begin{equation}\label{Gamma2}
\frac{g_{\mathfrak{i},y}(\textbf{a})}{f_{\mathfrak{i},x}(\textbf{b})}\ \leq\ \eta^k
\end{equation}
for all $\mathfrak{i}=(i_{1},\dots, i_{k})\in\mathcal{I}^k$ and all $\textbf{a}, \textbf{b}\in[0, 1]^2$. For the bottom left term direct expansion or induction gives 
\begin{equation}\label{gxsum}
g_{\mathfrak{i},x}(\textbf{a})= \sum_{j=1}^k G_j(\textbf{a})
\end{equation}
where, using the chain rule,
\begin{eqnarray}
G_j(\textbf{a})
&=&
 g_{i_1,y}(S_{i_{2}\cdots i_k}\textbf{a})\cdots g_{i_{j-1},y}(S_{i_{j}\cdots i_k}\textbf{a})
 g_{i_j,x}(S_{i_{j+1}\cdots i_k}\textbf{a})\nonumber\\ 
&& \hspace{3.5cm} \times f_{i_{j+1},x}(S_{i_{j+2}\cdots i_k}\textbf{a})\cdots f_{i_{k-1},x}(S_{i_k}\textbf{a})f_{i_k,x}(\textbf{a})\label{gxprod0}\\
&=&\bigg(\prod_{l=1}^{j-1}  g_{i_l,y}(S_{i_{l+1}\cdots i_k}\textbf{a})\bigg)  g_{i_j,x}(S_{i_{j+1}\cdots i_k}\textbf{a}) \bigg(\prod_{l=j+1}^{k} f_{i_l,x}(S_{i_{l+1}\cdots i_k}\textbf{a})\bigg)\label{gxprod}\\
&=& g_{i_1\cdots i_{j-1},y}(S_{i_{j}\cdots i_k}\textbf{a})g_{i_j,x}(S_{i_{j+1}\cdots i_k}\textbf{a})  f_{i_{j+1}\cdots i_k,x}(S_{i_{j+2}\cdots i_k}\textbf{a}).\label{gxprod2}
\end{eqnarray}

 The next two lemmas obtain estimates on the entries of \eqref{Der} that are uniform in $\mathfrak{i}$ and $\textbf{a}$.
\vspace{5mm}

\begin{lem}\label{New Lemma}
There exists a constant $R>0$ such that for all $\mathfrak{i}\in\mathcal{I}^*$ and all ${\bf a}, {\bf b}\in[0, 1]^2$,
\begin{equation}\label{ratiofandg}
R^{-1} \ \leq\ \frac{f_{\mathfrak{i},x}({\bf a})}{f_{\mathfrak{i},x}({\bf b})},\  \frac{g_{\mathfrak{i},y}({\bf a})}{g_{\mathfrak{i},y}({\bf b})}\ \leq \ R.
\end{equation}
\end{lem}
\begin{proof}
Note that since each $f_{i_j}$ is a $C^{1+\alpha}$ map there is a number $B$ such that
$$
|f_{i,x}(\textbf{a}^{\prime})-f_{i,x}(\textbf{b}^{\prime})|\leq B |\textbf{a}^{\prime}-\textbf{b}^{\prime}|^\alpha
$$
for all $i\in\mathcal{I}$ and all  $\textbf{a}^{\prime}, \textbf{b}^{\prime}\in[0,1]^2$.
For $\mathfrak{i}=(i_{1},\dots, i_{k})\in\mathcal{I}^k$ and $\textbf{a}, \textbf{b}\in[0, 1]^2$,  identity \eqref{fgchain} gives
\begin{eqnarray}
\frac{f_{\mathfrak{i},x}(\textbf{a})}{f_{\mathfrak{i},x}(\textbf{b})}
&=& 
\prod_{j=1}^k \frac{f_{i_j,x}(S_{i_{j+1}\cdots i_k}\textbf{a})}{f_{i_j,x}(S_{i_{j+1}\cdots i_k}\textbf{b})}\nonumber\\
&=& 
\prod_{j=1}^k \bigg( 1+ \frac{f_{i_j,x}(S_{i_{j+1}\cdots i_k}\textbf{a})- f_{i_j,x}(S_{i_{j+1}\cdots i_k}\textbf{b})}
{f_{i_j,x}(S_{i_{j+1}\cdots i_k}\textbf{b})}\bigg)\nonumber \\
&\leq& 
\prod_{j=1}^k \bigg( 1+ \frac{B|S_{i_{j+1}\cdots i_k}\textbf{a}- S_{i_{j+1}\cdots i_k}\textbf{b}|^\alpha}
{d}\bigg)\label{fix1}\\
&\leq& 
\prod_{j=1}^k \bigg( 1+ \frac{B\, c^{(k-j)\alpha} \, |\textbf{a}-\textbf{b}|^\alpha}
{d}\bigg)\nonumber\\
&\leq& 
\prod_{j=1}^k \exp\Big( \frac{2^\alpha B }{d} c^{(k-j)\alpha} \Big)\nonumber\\
&\leq& 
\exp\sum_{j=1}^k \Big( \frac{2^\alpha B}{d} c^{(k-j)\alpha} \Big)\nonumber\\
&\leq& 
\exp\Big( \frac{2^\alpha B}{d(1-c^\alpha)}\Big)\nonumber
\end{eqnarray}
using that $|\textbf{a}-\textbf{b}|\leq 2$. 
Setting $R = \exp({2^\alpha B/d(1-c^\alpha)})$ gives \eqref{ratiofandg} for $f_{\mathfrak{i},x}$,
with the left-hand estimate obtained by reversing the roles of $\textbf{a}$ and $\textbf{b}$. A similar argument using \eqref{fgchain} applies for  $g_{\mathfrak{i},y}$.
\end{proof}

We turn to the bottom left entries $g_{\mathfrak{i},x}$.
\bigskip

\begin{lem}\label{Off-DiagonalB}
There exists $C>0$ such that for all $\mathfrak{i}\in\mathcal{I}^*$ and all ${\bf a},{\bf b}\in [0, 1]^2$
\begin{equation}\label{ratio1}
\bigg| \frac{g_{\mathfrak{i},x}({\bf a})}{ f_{\mathfrak{i},x}({\bf b})}\bigg| \leq C. 
\end{equation}
\end{lem}

\begin{proof}
Let $\mathfrak{i}=(i_{1},\dots, i_{k})\in\mathcal{I}^k$ and $\textbf{a}, \textbf{b}\in[0, 1]^2$. Then for $1\leq j\leq k$ identities \eqref{fgchain} and  \eqref{gxprod} give
\begin{eqnarray*}
\bigg|\frac{G_j(\textbf{a})}{f_{\mathfrak{i},x}(\textbf{b})}\bigg|
&=& 
\bigg|\bigg( \prod_{l=1}^{j-1} \frac{g_{i_l,y}(S_{i_{l+1}\cdots i_k}\textbf{a})}{f_{i_l,x}(S_{i_{l+1}\cdots i_k}\textbf{b})}\bigg) 
\frac{g_{i_j,x}(S_{i_{j+1}\cdots i_k}\textbf{a})}{f_{i_j,x}(S_{i_{j+1}\cdots i_k}\textbf{b})} 
\bigg(\prod_{l=j+1}^{k} \frac{f_{i_l,x}(S_{i_{l+1}\cdots i_k}\textbf{a})}{f_{i_l,x}(S_{i_{l+1}\cdots i_k}\textbf{b})}\bigg)\bigg|\\
&=& 
\bigg( \prod_{l=1}^{j-1}\bigg|  \frac{g_{i_l,y}(S_{i_{l+1}\cdots i_k}\textbf{a})}{f_{i_l,x}(S_{i_{l+1}\cdots i_k}\textbf{b})}\bigg|\bigg)  
\bigg| \frac{g_{i_j,x}(S_{i_{j+1}\cdots i_k}\textbf{a})}{f_{i_j,x}(S_{i_{j+1}\cdots i_k}\textbf{b})}\bigg| 
\bigg|\frac{f_{i_{j+1}\cdots i_k,x}(\textbf{a})}{f_{i_{j+1}\cdots i_k,x}(\textbf{b})}\bigg|\\
&\leq & \eta^{j-1} \eta R,
\end{eqnarray*}
using  \eqref{Gamma1} and where $R$ is as in \eqref{ratiofandg}.
Hence by \eqref{gxsum}
$$\bigg| \frac{g_{\mathfrak{i},x}(\textbf{a})}{ f_{\mathfrak{i},x}(\textbf{b})}\bigg|
\ =\ \bigg| \frac{\sum_{j=1}^k G_j(\textbf{a}) }{ f_{\mathfrak{i},x}(\textbf{b})}\bigg|
\ \leq	 \  \sum_{j=1}^k\bigg|\frac{G_j(\textbf{a})}{ f_{\mathfrak{i},x}(\textbf{b})}\bigg| \ \leq \  \sum_{j=1}^k R\,\eta^j 
\ <\ \frac{R\,\eta }{1-\eta},$$
giving \eqref{ratio1} with $C= R\eta/(1-\eta)$.
\end{proof}
\vspace{5mm}

Recall that the singular values of an $n\times n$ matrix $A$  are defined to be the eigenvalues of $A^TA$. For $\textbf{a}=(a_1, a_2)\in [0, 1]^2$ write $\alpha_1(D_{\textbf{a}} S_{\mathfrak{i}})\geq \alpha_2(D_{\textbf{a}} S_{\mathfrak{i}})$ for the singular values of $D_{\textbf{a}} S_{\mathfrak{i}}$. 
\medskip

\begin{lem}\label{Singular Value Bound}
The singular values of the Jacobian matrices $D_{{\bf a}}S_{\mathfrak{i}}$ satisfy
\begin{equation}\label{alpha1}
\alpha_1(D_{{\bf a}}S_{\mathfrak{i}})\asymp f_{\mathfrak{i}, x}({\bf a})
\end{equation}
and
\begin{equation}\label{alpha2}
\alpha_2(D_{{\bf a}}S_{\mathfrak{i}})\asymp g_{\mathfrak{i}, y}({\bf a})
\end{equation}
for all ${\bf a}\in[0, 1]^2$ and $\mathfrak{i}\in\mathcal{I}^*$.
\end{lem}

\begin{proof}
Let 
\[
A=\begin{pmatrix} 
a & 0 \\
b & c
\end{pmatrix}.
\]
be a matrix with $0\leq c\leq a$ and $0\leq b\leq Ca$ for some constant $C>0$. Calculating the larger singular value $\alpha_1(A)$ of $A$, which is the (positive) square root of the larger eigenvalue of $AA^T$, 
$$
\alpha_1(A)^2 = {\textstyle \frac{1}{2}} \Big( (a^2 +b^2+c^2) + \big((a^2 +b^2+c^2)^2 -4a^2c^2\big)^{1/2}\Big).
$$
Making obvious estimates, 
$$
{\textstyle \frac{1}{2}} a^2\ \leq\ \alpha_1(A)^2 \ \leq \  a^2 +b^2+c^2\ \leq\ (2+C^2)a^2.
$$
Applying this to the matrix
$$
D_{\textbf{a}} S_{\mathfrak{i}} = \begin{pmatrix} 
f_{\mathfrak{i},x}(\textbf{a}) & 0 \\
g_{\mathfrak{i},x}(\textbf{a}) & g_{\mathfrak{i},y}(\textbf{a})
\end{pmatrix},
$$
where $0 \leq g_{\mathfrak{i},y}(\textbf{a}) \leq f_{\mathfrak{i},x}(\textbf{a})$ by \eqref{Gamma2} and $0\leq |g_{\mathfrak{i},x}(\textbf{a})| \leq C f_{\mathfrak{i},x}(\textbf{a})$ by \eqref{ratio1}, gives \eqref{alpha1}. Using that $\alpha_1(A)\alpha_2(A) = \det A=ac$ for the matrix $A$, \eqref{alpha2} follows immediately from \eqref{alpha1}.
\end{proof}

A immediate  consequence of Lemmas \ref{New Lemma} and \ref{Singular Value Bound}  is that the singular values of the Jacobian matrices satisfy
\begin{equation}\label{Bounded Distortion}
\alpha_1(D_{\textbf{a}} S_{\mathfrak{i}})\asymp \alpha_1(D_{\textbf{b}} S_{\mathfrak{i}})\quad \mbox{ and } \quad
\alpha_2(D_{\textbf{a}} S_{\mathfrak{i}})\asymp \alpha_2(D_{\textbf{b}} S_{\mathfrak{i}})
\end{equation}
for all $\textbf{a}, \textbf{b}\in[0, 1]^2$ and  $\mathfrak{i}\in\mathcal{I}^*$. 


We define the projection map onto the $x$-axis $\pi:\mathbb{R}^2\rightarrow\mathbb{R}$ by $\pi(x, y)=x$.  It is immediate that the projection of the nonlinear measure $\mu$ onto the $x$-axis, $\pi(\mu)$, is a self-conformal measure. It follows from a result of Peres and Solomyak \cite{Peres} that the $L^q$-spectra of this projected measure, which we denote by 
\begin{equation}\label{betaq}
\beta(q):=\tau_{\pi(\mu)}(q),
\end{equation}
exists for $q\geq 0$. Note that this  holds even if there are complicated overlaps between the components of the projected measure, which is the typical situation for us.

For $s\in\mathbb{R}$, $q\geq 0$ and $\textbf{a}\in[0, 1]^2$, we define the \textit{q-modified singular value function}, $\psi^{s, q}_{\textbf{a}}:\mathcal{I}^*\rightarrow (0,\infty)$  by
\begin{equation}\label{qmodsv}
\psi^{s, q}_{\textbf{a}}(\mathfrak{i}) = p(\mathfrak{i})^q\alpha_1(D_{\textbf{a}} S_{\mathfrak{i}})^{\beta(q)}\alpha_2(D_{\textbf{a}} S_{\mathfrak{i}})^{s-\beta(q)}.
\end{equation}
It follows from (\ref{Bounded Distortion}) that for all $\textbf{a}, \textbf{b}\in[0, 1]^2$ and $\mathfrak{i}\in\mathcal{I}^*$ we have $\psi^{s, q}_{\textbf{a}}(\mathfrak{i})\asymp_{s,q}\psi^{s, q}_{\textbf{b}}(\mathfrak{i})$. Moreover, by Lemma \ref{New Lemma}, 
\begin{equation}\label{psiequiv}
\psi^{s, q}_{\textbf{a}}(\mathfrak{i}) \asymp_{s,q} p(\mathfrak{i})^qf_{\mathfrak{i}, x}(\textbf{a})^{\beta(q)}g_{\mathfrak{i}, y}(\textbf{a})^{s-\beta(q)}.
\end{equation}

For each $k\in\mathbb{N}$ define $\Psi_{\textbf{a}, k}^{s, q}$ by
\begin{equation}\label{Psi}
\Psi_{\textbf{a}, k}^{s, q}=\sum_{\mathfrak{i}\in\mathcal{I}^k}\psi^{s, q}_{\textbf{a}}(\mathfrak{i}).
\end{equation}
The quantities $\psi^{s, q}_{\textbf{a}}(\mathfrak{i})$ and $\Psi_{\textbf{a}, k}^{s, q}$ satisfy some useful multiplicative properties, similar to those from \cite[Lemma 2.2]{Fraser}.
\vspace{5mm}
\begin{lem}\label{sub}
Let $s\in\mathbb{R}$, $q\geq 0$ and ${\bf a}\in[0, 1]^2$.

{\rm (a)} If $\mathfrak{i},\mathfrak{j}\in\mathcal{I}^*$ then
\begin{equation}\label{pta}
\psi^{s, q}_{{\bf a}}(\mathfrak{i}\mathfrak{j})\asymp_{s,q}\psi^{s, q}_{{\bf a}}(\mathfrak{i})\psi^{s, q}_{{\bf a}}(\mathfrak{j}).
\end{equation}
{\rm (b)} If $k,l\in\mathbb{N}$ then
\begin{equation}\label{ptb}
\Psi_{{\bf a}, k+l}^{s, q}\asymp_{s,q} \Psi_{{\bf a}, k}^{s, q}\Psi_{{\bf a}, l}^{s, q}.
\end{equation}
\end{lem}

\begin{proof}
By the chain rule applied to $f_{\mathfrak{ij}, x}$ and using \eqref{ratiofandg}, 
$$f_{\mathfrak{ij}, x}({\bf a}) = f_{\mathfrak{i}, x}(S_\mathfrak{i}{\bf a})f_{\mathfrak{j}, x}({\bf a})
\asymp_{s,q} f_{\mathfrak{i}, x}({\bf a})f_{\mathfrak{j}, x}({\bf a}), $$
and similarly 
$$g_{\mathfrak{ij}, y}({\bf a}) \asymp_{s,q} g_{\mathfrak{i}, y}({\bf a})g_{\mathfrak{j}, y}({\bf a}). $$
Using the form \eqref{psiequiv}
\begin{align*}
\psi^{s, q}_{\textbf{a}}(\mathfrak{ij}) &\asymp_{s,q} p(\mathfrak{ij})^qf_{\mathfrak{ij}, x}(\textbf{a})^{\beta(q)}g_{\mathfrak{ij}, y}(\textbf{a})^{s-\beta(q)}\\
&\asymp_{s,q} p(\mathfrak{i})^q p(\mathfrak{j})^q f_{\mathfrak{i}, x}(\textbf{a})^{\beta(q)}f_{\mathfrak{j}, x}(\textbf{a})^{\beta(q)}g_{\mathfrak{i}, y}(\textbf{a})^{s-\beta(q)}g_{\mathfrak{j}, y}(\textbf{a})^{s-\beta(q)}\\
&\asymp_{s,q} \psi^{s, q}_{\textbf{a}}(\mathfrak{i})\psi^{s, q}_{\textbf{a}}(\mathfrak{j}) 
\end{align*}
giving \eqref{pta}

For part (b), if $k, l\in\mathbb{N}$ then
\[
\Psi_{\textbf{a}, k+l}^{s, q}=\sum_{\mathfrak{i}\in\mathcal{I}^{k+l}}\psi^{s, q}_{\textbf{a}}(\mathfrak{i})=\sum_{\mathfrak{i}\in\mathcal{I}^{k}}\sum_{\mathfrak{j}\in\mathcal{I}^{l}}\psi^{s, q}_{\textbf{a}}(\mathfrak{ij})
\]
and
\[
\Psi_{\textbf{a}, k}^{s, q}\Psi_{\textbf{a}, l}^{s, q}=\Big(\sum_{\mathfrak{i}\in\mathcal{I}^{k}}\psi^{s, q}_{\textbf{a}}(\mathfrak{i})\Big)\Big(\sum_{\mathfrak{j}\in\mathcal{I}^{l}}\psi^{s, q}_{\textbf{a}}(\mathfrak{j})\Big)=\sum_{\mathfrak{i}\in\mathcal{I}^{k}}\sum_{\mathfrak{j}\in\mathcal{I}^{l}}\psi^{s, q}_{\textbf{a}}(\mathfrak{i})\psi^{s, q}_{\textbf{a}}(\mathfrak{j}).
\]
Applying part (a) to the double sums completes the proof.
\end{proof}

We call a sequence $\{a_n\}_{n\in\mathbb{N}}$ with $a_n>0$ (such as those in Lemma \ref{sub}) for which there exists an absolute constants $0<K_1 \leq K_2$ such that 
\begin{equation}\label{Almost}
K_1 a_na_m \leq a_{n+m}\leq K_2 a_na_m
\end{equation}
for all $n,m \in\mathbb{N}$  \textit{almost-multiplicative}. 
For such sequences the limit $\lim_{n\to\infty} a_n^{1/n}$ exists, see for example \cite[Corollary 1.2]{Falconer Techniques}.

It follows from Lemma \ref{sub} that for each $\textbf{a}\in[0, 1]^2$ we may define a function $P_{\textbf{a}}:\mathbb{R}\times[0, \infty)\rightarrow[0, \infty)$ by
\[
P_{\textbf{a}}(s, q)=\lim_{k\rightarrow\infty}(\Psi_{\textbf{a}, k}^{s, q})^{1/k}.
\]
Note that the  value of $P_{\textbf{a}}(s, q)$ is unchanged if we replace the right-hand side of \eqref{qmodsv} by 
the right-hand side of \eqref{psiequiv} in the definition of $\psi^{s, q}_{\textbf{a}}(\mathfrak{i})$ and thus of $\Psi_{\textbf{a}, k}^{s, q}$. Moreover, as $\psi^{s, q}_{\textbf{a}}(\mathfrak{i})\asymp_{s,q}\psi^{s, q}_{\textbf{b}}(\mathfrak{i})$ and thus $\Psi_{\textbf{a}, k}^{s, q}\asymp_{s,q} \Psi_{\textbf{b}, k}^{s, q}$ for all $\textbf{a}, \textbf{b}\in[0, 1]^2$ it is easy to see that $P_{\textbf{a}}$ is independent of the choice of $\textbf{a}$. Thus we shall just write $P$ instead of $P_{\textbf{a}}$.  For a fixed $q \geq 0$ we think of the function  $s\mapsto \log P(s,q)$ as the \emph{topological pressure} of the system.

We also write the following
\begin{align*}
\alpha_{\textnormal{min}}&=\inf\{\alpha_2(D_{\textbf{a}} S_{i}):\textbf{a}\in[0, 1]^2,\  i\in\mathcal{I}\},\\
\alpha_{\textnormal{max}}&=\sup\{\alpha_1(D_{\textbf{a}} S_{i}):\textbf{a}\in[0, 1]^2, \ i\in\mathcal{I}\},\\
p_{\textnormal{min}}&=\textnormal{min}\{p_i: i\in\mathcal{I}\},\\
p_{\textnormal{max}}&=\textnormal{max}\{p_i: i\in\mathcal{I}\}
\end{align*}
and note that $0<\alpha_{\textnormal{min}}, \alpha_{\textnormal{max}},p_{\textnormal{min}},p_{\textnormal{max}} <1$.

Recall that the $L^q$-spectrum of a given measure is Lipschitz continuous (as it is concave and decreasing) on $[\lambda, \infty)$ for all $\lambda>0$. Let $L_{\lambda}$ denote the Lipschitz constant of $\beta$ on $[\lambda, \infty)$. We can now state some basic properties of   $P$.

\vspace{5mm}

\begin{lem}\label{Pressure Properties}

(1)   For $s, r\in\mathbb{R}$ and $\lambda>0$ define
$$
U (s, r, \lambda)=
\textnormal{min}\left\{\alpha_{\textnormal{min}}^s p_{\textnormal{min}}^r, \alpha_{\textnormal{min}}^s p_{\textnormal{max}}^r, \alpha_{\textnormal{max}}^s p_{\textnormal{min}}^r, \alpha_{\textnormal{max}}^s p_{\textnormal{max}}^r\right\}(\alpha_{\textnormal{max}}/\alpha_{\textnormal{min}})^{\textnormal{min}\{-L_{\lambda}r, 0\}}
$$
and
$$
V (s, r, \lambda)= 
\textnormal{max}\left\{\alpha_{\textnormal{min}}^s p_{\textnormal{min}}^r, \alpha_{\textnormal{min}}^s p_{\textnormal{max}}^r, \alpha_{\textnormal{max}}^s p_{\textnormal{min}}^r, \alpha_{\textnormal{max}}^s p_{\textnormal{max}}^r\right\}(\alpha_{\textnormal{max}}/\alpha_{\textnormal{min}})^{\textnormal{max}\{-L_{\lambda}r, 0\}}.
$$
Then for all $s,t\in\mathbb{R}, \lambda>0, q\geq\lambda$ and $r\geq\lambda-q$
\[
U (s, r, \lambda)P(t, q)\leq P(s+t, q+r)\leq V(s, r, \lambda)P(t, q),
\]
and for all $s, t\in\mathbb{R}$
\[
\textnormal{min}\left\{\alpha_{\textnormal{min}}^s, \alpha_{\textnormal{max}}^s\right\}P(t, 0)\leq P(s+t, 0)\leq \textnormal{max}\left\{\alpha_{\textnormal{min}}^s, \alpha_{\textnormal{max}}^s\right\}P(t, 0).
\]
Also for all $s\in\mathbb{R}$ and $q\geq 0$
\[
P(s, q)\leq p_{\textnormal{max}}^qP(s, 0).
\]

(2) $P$ is continuous on $\mathbb{R}\times(0, \infty)$ and on $\mathbb{R}\times\{0\};$

(3) $P$ is strictly decreasing in $s\in\mathbb{R}$ and $q\in(0, \infty);$

(4) For each $q\geq 0$ there exists a unique $s\geq 0$ such that $P(s, q)=1$.
\end{lem}

\begin{proof}
This is essentially the same as the proof of the analogous result of Fraser \cite[Lemma 2.3]{Fraser} and as such is omitted.
\end{proof}

It follows from Lemma \ref{Pressure Properties} that we may define a function $\gamma:[0, \infty)\rightarrow\mathbb{R}$ by $P(\gamma(q), q)=1$ which we shall refer to as a \textit{moment scaling function}.   The moment scaling function satisfies the following useful properties.

\vspace{5mm}

\begin{lem} \,

(1) $\gamma$ is strictly decreasing on $[0, \infty);$

(2) $\gamma$ is continuous on $(0, \infty);$

(3) $\gamma(1)=0$ and $\lim_{q\rightarrow\infty}\gamma(q)=-\infty;$

(4) $\gamma$ is convex on $(0, \infty)$.
\end{lem}

\begin{proof}
This follows by the same reasoning as in the proof of \cite[Lemma 2.5]{Fraser}.
\end{proof}

We can now state our main theorem which relates $\gamma$ to the $L^q$-spectrum $\tau_{\mu}(q)$ of $\mu$. 

\vspace{5mm}

\begin{thm}\label{Main Theorem}
Let $\mu$ be a nonlinear measure which satisfies the domination condition \eqref{Domination}. Then

(1) For $q\in [0, 1]$
\[
\overline{\tau}_{\mu}(q)\leq\gamma(q);
\]
(2) For $q\geq 1$
\[
\underline{\tau}_{\mu}(q)\geq\gamma(q);
\]
(3) If $\mu$ also satisfies the ROSC then for all $q\geq 0$
\[
\tau_{\mu}(q)=\gamma(q).
\]
\end{thm}

We shall prove this theorem in  Section \ref{Calculating Section}.

As a corollary we are able  to calculate the box dimension of the support of these measures. We recall the definition of the box dimension.

\vspace{5mm}

\begin{defn}\label{Box Dimension Def}
Let $X\subset\mathbb{R}^n$ be bounded and non-empty and let $N_{\delta}(X)$ denote the minimal number of sets of diameter at most $\delta$ needed to cover $X$. The upper and lower box dimension are defined to be 
\[
\overline{\dim_{B}} X= \overline{\lim}_{\delta\rightarrow 0} \frac{\log N_{\delta}(X)}{-\log \delta}
\]
and
\[
\underline{\dim_{B}} X= \underline{\lim}_{\delta\rightarrow 0} \frac{\log N_{\delta}(X)}{-\log \delta}
\]
respectively. If these numbers coincide then we define the box dimension of $X$, denoted $\dim_B X$, to be their common value.
\end{defn}

\vspace{5mm}

\begin{cor}\label{Box Dimension}
Let $F$ be a nonlinear attractor  which satisfies the domination condition \eqref{Domination}. Then
(1) 
\[
\overline{\dim_{B}} F \leq \gamma(0);
\]
(2) If $F$ also satisfies the ROSC then 
\[
\dim_{B} F =\gamma(0).
\]

\end{cor}

\begin{proof}
It is well-known that the upper and lower box dimension of the support of a measure is given by the upper and lower $L^q$-spectrum at $0$. The result is then immediate from Theorem \ref{Main Theorem}.
\end{proof}

Note that $\gamma(0)$ depends on $\beta(0)=\dim_B \pi F$, the box dimension of the projection of $F$ onto the $x$-axis.  Also note that by standard results, e.g. \cite[Corollary 3.10]{Falconer}, the \emph{packing} dimension of a nonlinear attractor coincides with the upper box dimension and so Corollary \ref{Box Dimension} also yields the packing dimension.

\section{Calculating the $L^q$-spectrum}\label{Calculating Section}

We begin this section by introducing some notation. 
For ${\mathfrak{i}}=(i_1, i_2,\dots, i_{k}) \in \mathcal{I}^*$ let   $\widehat{\mathfrak{i}}\in\mathcal{I}^*\cup\{\omega\}$ be given by
\[
\hat{\mathfrak{i}}=(i_1, i_2,\dots, i_{k-1})
\]
where $\omega$ is the empty word. For $\delta\in(0, 1]$ and $\textbf{a}\in[0, 1]^2$ we define the $\delta$\textit{-stopping} by
\[
\mathcal{I}_{\textbf{a},\delta}=\{\mathfrak{i}\in\mathcal{I}^*:\alpha_2(D_{\textbf{a}} S_{\mathfrak{i}})<\delta\leq\alpha_2(D_{\textbf{a}} S_{\hat{\mathfrak{i}}})\}
\]
where  $S_{\omega}$ is the identity map. Note that if $\mathfrak{i}\in\mathcal{I}_{\textbf{a},\delta}$ then
\begin{equation}\label{Alpha Min Bound}
\alpha_{\textnormal{min}}\delta\leq\alpha_2(D_{\textbf{a}} S_{\mathfrak{i}})<\delta.
\end{equation}
For $\mathfrak{i}\in\mathcal{I}^*$ let $\mu_{\mathfrak{i}}=p(\mathfrak{i})\mu\circ S_{\mathfrak{i}}^{-1}$ and $F_{\mathfrak{i}}=S_{\mathfrak{i}}(F)=\textnormal{supp}\mu_{\mathfrak{i}}.$ Note that for all $\textbf{a}\in[0, 1]^2$ and $ \delta \in (0,1]$,
\[
\mu=\sum_{\mathfrak{i}\in\mathcal{I}_{\textbf{a},\delta}}\mu_{\mathfrak{i}}.
\]

\begin{lem}\label{Sum Lemma}
Let ${\bf a}\in[0, 1]^2$, $t\in\mathbb{R}$ and $q\geq 0$.

(1) If $t>\gamma(q)$ then
\[
\sum_{\mathfrak{i}\in\mathcal{I}_{{\bf a},\delta}}\psi^{t, q}_{{\bf a}}(\mathfrak{i})\lesssim_{t, q} 1
\]
for all $\delta\in(0, 1]$.

(2) If $t<\gamma(q)$ then
\[
\sum_{\mathfrak{i}\in\mathcal{I}_{{\bf a},\delta}}\psi^{t, q}_{{\bf a}}(\mathfrak{i})\gtrsim_{t, q} 1
\]
for all $\delta\in(0, 1]$.

\end{lem}

\begin{proof}
The proof follows that of \cite[Lemma 7.1]{Fraser} which only depends on the multiplicative properties of $\Psi$ (which we have established here) so is omitted.
\end{proof}

\vspace{5mm}

Our next lemma allows us to control the length of the side of $S_{\mathfrak{i}}([0, 1]^2)$ in terms of the length of its base.

\begin{lem}\label{Slope Ratio}
There exists $L>0$ such that for all $\mathfrak{i}\in\mathcal{I}^*$ and all $0\leq a<b\leq 1$ 
\[
\frac{|g_{\mathfrak{i}}(b, 1)-g_{\mathfrak{i}}(a,1)|}{|f_{\mathfrak{i}}(b)-f_{\mathfrak{i}}(a)|}\leq L,
\]
\end{lem}
noting that $f_{\mathfrak{i}}$ depends only on the first coordinate of its argument.
\begin{proof}
By the mean value theorem there exist $c_1, c_2\in (a,b)$ such that
\[
\frac{|g_{\mathfrak{i}}(b, 1)-g_{\mathfrak{i}}(a,1)|}{|f_{\mathfrak{i}}(b)-f_{\mathfrak{i}}(a)|}=\frac{|g_{\mathfrak{i},x}(c_1, 1)||b-a|}{|f_{\mathfrak{i},x}(c_2)||b-a|}=\frac{|g_{\mathfrak{i},x}(c_1, 1)|}{|f_{\mathfrak{i},x}(c_2)|} \lesssim 1
\]
by Lemma \ref{Off-DiagonalB}.
\end{proof}


The standard inequalities, that if $k\in\mathbb{N}, a_1,\dots, a_k\geq 0$ and $q\geq 0$ then
\begin{equation}\label{Moment Sums Lemma1}
\left(\sum_{i=1}^k a_i\right)^q\asymp_{k, q}\sum_{i=1}^k a_i^q,
\end{equation}
 will be helpful when manipulating moment sums.

Recall from \eqref {momsum} that $\mathcal{D}_{\delta}^q$  denotes the $q$-th power moment sum of a measure over the $\delta$-mesh cubes $\mathcal{D}_{\delta}$. Our next result compares the moment sums of $\mu_{\mathfrak{i}}=p(\mathfrak{i})\mu\circ S_{\mathfrak{i}}^{-1}$ on $S_{\mathfrak{i}}(F)$ with moment sums of the projection of $\mu$ onto the horizontal axis. This is analogous to \cite[Lemma 7.2]{Fraser} but in the nonlinear case more care is needed. 

\vspace{5mm}

\begin{lem}\label{Moment Sums Lemma2}
For each $q\geq 0$ and  ${\bf a}\in[0, 1]^2$,  there exist numbers $\widehat{A}, \widehat{B}>0$ such that if we write
\begin{equation}\label{A Hat}
\widehat{A}_{\mathfrak{i},\delta}=\frac{\widehat{A}\delta}{\alpha_1(D_{{\bf a}}S_{\mathfrak{i}})}
\qquad {\text and } \qquad
\widehat{B}_{\mathfrak{i},\delta}=\frac{\widehat{B}\delta}{\alpha_1(D_{{\bf a}}S_{\mathfrak{i}})}
\end{equation}
for $\delta\in (0, 1]$ and $\mathfrak{i}\in\mathcal{I}_{{\bf a},\delta}$ 
then
\begin{equation}\label{hatcomp}
\mathcal{D}^q_{\widehat{B}_{\mathfrak{i},\delta}}(p(\mathfrak{i})\pi\mu)\ \lesssim\ \mathcal{D}^q_{\delta}(\mu_{\mathfrak{i}})\ \lesssim\  \mathcal{D}^q_{\widehat{A}_{\mathfrak{i},\delta}}(p(\mathfrak{i})\pi\mu).
\end{equation}

\end{lem}

\begin{proof}
As $\mathfrak{i}\in\mathcal{I}_{\textbf{a},\delta}$ we have $\alpha_2(D_{\textbf{a}} S_{\mathfrak{i}})<\delta$. We shall show that there are at most a constant number $k$ squares of the $\delta$-mesh that intersect $S_{\mathfrak{i}}([0, 1]^2)\supseteq \textnormal{supp}\mu_{\mathfrak{i}}$ in any vertical column of mesh squares.


For this we estimate the height of the intersection of $S_{\mathfrak{i}}([0, 1]^2)$ with a given vertical strip of width $\delta$. Note that for any such vertical strip there exists some $0\leq a < b\leq 1$ such that
\begin{equation}\label{Delta Bound}
|f_{\mathfrak{i}}(b)-f_{\mathfrak{i}}(a)|=\delta
\end{equation}
apart from at most two vertical strips (at the left and right ends of $S_{\mathfrak{i}}([0, 1]^2)$) for which 
\begin{equation}\label{Delta Bound1}
|f_{\mathfrak{i}}(b)-f_{\mathfrak{i}}(a)|\leq\delta
\end{equation}
with one of $a$ or $b$ equal to either 0 or 1. Then, for $a', b' \in [a,b]$,
\begin{eqnarray}
|g_{\mathfrak{i}}(b', 1)-g_{\mathfrak{i}}(a', 0)|
&\leq& |g_{\mathfrak{i}}(b', 1)-g_{\mathfrak{i}}(a', 1)|+|g_{\mathfrak{i}}(a', 1)-g_{\mathfrak{i}}(a', 0)|\label{giest}\\
&\leq&  L|f_{\mathfrak{i}}(b')-f_{\mathfrak{i}}(a')| + |g_{\mathfrak{i}, y}(a',c)| \nonumber\\
&\lesssim& L\delta +  \alpha_2(D_{(a',c)} S_{\mathfrak{i}})\nonumber \\
&\lesssim& \delta,\nonumber
\end{eqnarray}
where we have estimated the first term of \eqref{giest} using Lemma \ref{Slope Ratio} and (\ref{Delta Bound})-(\ref{Delta Bound1})  and the second term using the mean value theorem with $c\in (0, 1)$ followed by  \eqref{alpha2}, \eqref{Bounded Distortion} and \eqref{Alpha Min Bound}.

\begin{figure}[H]
  \centering
\includegraphics[width=6cm]{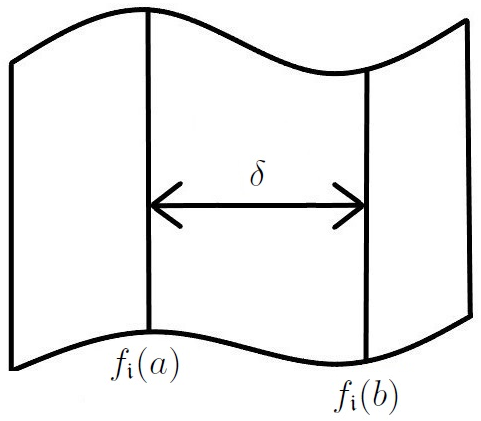}
\caption{$S_{\mathfrak{i}}([0, 1]^2)$ together with two points $f_{\mathfrak{i}}(a)$ and $f_{\mathfrak{i}}(b)$ which together form a vertical strip of width $\delta$.} 
\end{figure}

We have shown that the height of the intersection of $S_{\mathfrak{i}}([0, 1]^2)$ with every vertical strip with base length $\delta$ is at most $k' \delta$, where $k'$ is independent of $\mathfrak{i}$. Thus at most $k=\lceil  k' \rceil +1$ squares in any column of the $\delta$-mesh intersect $S_{\mathfrak{i}}([0, 1]^2)$
so using (\ref{Moment Sums Lemma1}) $\mathcal{D}^q_{\delta}(\mu_{\mathfrak{i}})\asymp \mathcal{D}^q_{\delta}(\pi\mu_{\mathfrak{i}})$ where $\pi\mu_{\mathfrak{i}}$ is the projection of $\mu_{\mathfrak{i}}$ onto the $x$-axis. In terms of the projection of  pre-images of the intersection of $\delta$-mesh cubes $Q$ with $S_{\mathfrak{i}}([0, 1]^2)$,
\begin{align}
\mathcal{D}^q_{\delta}(\mu_{\mathfrak{i}})&=\sum_{Q\in\mathcal{D}_{\delta}}\mu_{\mathfrak{i}}\left(Q\cap S_{\mathfrak{i}}([0, 1]^2)\right)^q\nonumber\\
&= p(\mathfrak{i})^q\sum_{Q\in\mathcal{D}_{\delta}}\mu\left(S_{\mathfrak{i}}^{-1}\left(Q\cap S_{\mathfrak{i}}([0, 1]^2)\right)\right)^q\nonumber\\
&\asymp p(\mathfrak{i})^q\sum_{Q\in\mathcal{D}_{\delta}}\pi\mu\left(\pi S_{\mathfrak{i}}^{-1}\left(Q\cap S_{\mathfrak{i}}([0, 1]^2)\right)\right)^q.\label{Moment Sum Calculation}
\end{align}

If $Q$ is a  $\delta$-mesh cube that intersects $S_{\mathfrak{i}}([0, 1]^2)$ other than one overlapping its left or right edge,  then $\pi S_{\mathfrak{i}}^{-1}\left(Q\cap S_{\mathfrak{i}}([0, 1]^2)\right)$  is an interval in $\mathbb{R}$ of length $\delta$. Writing $\hat{a}, \hat{b}$ for the endpoints of  $\pi S_{\mathfrak{i}}^{-1}\left(Q\cap S_{\mathfrak{i}}([0, 1]^2)\right)$ then  
using the mean value theorem and \eqref{alpha1},
\[
|\hat{a}-\hat{b}|\ =\ \frac{|f_{\mathfrak{i}}(\hat{a})-f_{\mathfrak{i}}(\hat{b})|}{|f_{\mathfrak{i},x} (\hat{c})|}\ = \ \frac{\delta}{|f_{\mathfrak{i},x} (\hat{c})|}\ \asymp\ \frac{\delta}{\alpha_1(D_{\textbf{a}}S_{\mathfrak{i}})},
\]
where $\hat{c}\in[0, 1]$. 


If $Q$ is one of the two cubes at the left or right end of $S_{\mathfrak{i}}([0, 1]^2)$ then we simply ``glue'' $\pi S_{\mathfrak{i}}^{-1}\left(Q\cap S_{\mathfrak{i}}([0, 1]^2)\right)$ to the adjacent interval (which will be on the right or left respectively). This will create a new interval which also has length comparable to $\delta/\alpha_1(D_{\textbf{a}}S_{\mathfrak{i}})$. 

Every projection of a  pre-image  $\pi S_{\mathfrak{i}}^{-1}\left(Q\cap S_{\mathfrak{i}}([0, 1]^2)\right)$ can be covered by an interval of length $\widehat{A}\delta/\alpha_1(D_{\textbf{a}}S_{\mathfrak{i}})$ and contains an interval of length $\widehat{B}\delta/\alpha_1(D_{\textbf{a}}S_{\mathfrak{i}})$, for some constants $\widehat{A} \geq \widehat{B}>0$. Recall the definitions \eqref{A Hat} of $\widehat{A}_{\mathfrak{i},\delta}$ and $\widehat{B}_{\mathfrak{i},\delta}$ and write $\mathcal{J}_{\delta}$ for the $\delta$-mesh on $\mathbb{R}$ centred at the origin. From (\ref{Moment Sum Calculation}), noting that each 
$J\in\mathcal{J}_{\widehat{A}_{\mathfrak{i},\delta}}$ can intersect $\pi S_{\mathfrak{i}}^{-1}\left(Q\cap S_{\mathfrak{i}}([0, 1]^2)\right)$ for at most $k\big\lceil \widehat{A}/\widehat{B}+1 \big\rceil$ many $Q\in\mathcal{D}_{\delta}$, and using \eqref{Moment Sums Lemma1},
\begin{align*}
\mathcal{D}^q_{\delta}(\mu_{\mathfrak{i}})&\asymp  p(\mathfrak{i})^q\sum_{Q\in\mathcal{D}_{\delta}}\pi\mu\left(\pi S_{\mathfrak{i}}^{-1}\left(Q\cap S_{\mathfrak{i}}([0, 1]^2)\right)\right)^q\\
&\lesssim p(\mathfrak{i})^q\sum_{J\in\mathcal{J}_{\widehat{A}_{\mathfrak{i},\delta}}}\pi\mu(J)^q\\
&=\sum_{J\in\mathcal{J}_{\widehat{A}_{\mathfrak{i},\delta}}}(p(\mathfrak{i})\pi\mu(J))^q\\
&=\mathcal{D}^q_{\widehat{A}_{\mathfrak{i},\delta}}(p(\mathfrak{i})\pi\mu).
\end{align*}
Similarly, each $\pi S_{\mathfrak{i}}^{-1}\left(Q\cap S_{\mathfrak{i}}([0, 1]^2)\right)$     intersects  at most $\big\lceil \widehat{A}/\widehat{B}+1 \big\rceil $ intervals  $J\in\mathcal{J}_{\widehat{B}_{\mathfrak{i},\delta}}$, and each interval $J\in\mathcal{J}_{\widehat{B}_{\mathfrak{i},\delta}}$  intersects $\pi S_{\mathfrak{i}}^{-1}\left(Q\cap S_{\mathfrak{i}}([0, 1]^2)\right)$ for at most $2k$  sets $Q\in\mathcal{D}_{\delta}$, so
\begin{align*}
\mathcal{D}^q_{\delta}(\mu_{\mathfrak{i}})&\asymp  p(\mathfrak{i})^q\sum_{Q\in\mathcal{D}_{\delta}}\pi\mu\left(\pi S_{\mathfrak{i}}^{-1}\left(Q\cap S_{\mathfrak{i}}([0, 1]^2)\right)\right)^q\\
&\gtrsim p(\mathfrak{i})^q\sum_{J\in\mathcal{J}_{\widehat{B}_{\mathfrak{i},\delta}}}\pi\mu(J)^q\\
&=\sum_{J\in\mathcal{J}_{\widehat{B}_{\mathfrak{i},\delta}}}(p(\mathfrak{i})\pi\mu(J))^q\\
&=\mathcal{D}^q_{\widehat{B}_{\mathfrak{i},\delta}}(p(\mathfrak{i})\pi\mu),
\end{align*}
giving the result.
\end{proof}

Notice that a simple consequence of the Definition \ref{lq} of the $L^q$-spectrum is that for all $\varepsilon>0, q\geq 0, p>0$ and $0< \delta \leq 1$
\begin{equation}\label{Moment Sum Bound}
p^q\delta^{-\beta(q)+\varepsilon/2}\lesssim_{\varepsilon,q}\mathcal{D}^q_{\delta}(p\pi\mu)\lesssim_{\varepsilon,q}p^q\delta^{-\beta(q)-\varepsilon/2}
\end{equation}
We now turn to proving our main result, Theorem \ref{Main Theorem}.

\begin{proof}[Proof of Theorem \ref{Main Theorem}]
The first two parts of this proof follow Fraser's proof of \cite[Theorem 2.6]{Fraser} but we reproduce it here due to its centrality to our result. 

{\it Part (1)}. Let $q\in [0, 1]$ and let $\delta\in (0, 1]$ and  $\textbf{a}\in[0, 1]^2$. It is sufficient to show that $\overline{\tau}_{\mu}(q)\leq \gamma(q)$. As $q\in [0, 1]$,
\begin{align*}
\mathcal{D}^q_{\delta}(\mu)=\sum_{Q\in\mathcal{D}_{\delta}}\mu(Q)^q=\sum_{Q\in\mathcal{D}_{\delta}}\Big(\sum_{\mathfrak{i}\in\mathcal{I}_{\delta}}\mu_{\mathfrak{i}}(Q)\Big)^q\leq \sum_{Q\in\mathcal{D}_{\delta}}\sum_{\mathfrak{i}\in\mathcal{I}_{\delta}}\mu_{\mathfrak{i}}(Q)^q
&=\sum_{\mathfrak{i}\in\mathcal{I}_{\delta}}\sum_{Q\in\mathcal{D}_{\delta}}\mu_{\mathfrak{i}}(Q)^q \\ \\
&=\sum_{\mathfrak{i}\in\mathcal{I}_{\delta}}\mathcal{D}^q_{\delta}(\mu_{\mathfrak{i}}).
\end{align*}

 Thus for all $\varepsilon>0$,
\begin{align*}
\delta^{\gamma(q)+\varepsilon}\mathcal{D}_{\delta}^q(\mu)\ &\leq\ \delta^{\gamma(q)+\varepsilon}\sum_{\mathfrak{i}\in\mathcal{I}_{\delta}}\mathcal{D}^q_{\delta}(\mu_{\mathfrak{i}})&\\
&\lesssim\  \delta^{\gamma(q)+\varepsilon}\sum_{\mathfrak{i}\in\mathcal{I}_{\delta}}\mathcal{D}^q_{\widehat{A}_{\mathfrak{i},\delta}}(p(\mathfrak{i})\pi_{\mathfrak{i}}\mu)&\textnormal{by \eqref{hatcomp}}\\
&\lesssim_{\varepsilon,q}\ \delta^{\gamma(q)+\varepsilon}\sum_{\mathfrak{i}\in\mathcal{I}_{\delta}}p(\mathfrak{i})^q\left(\frac{\widehat{A}\delta}{\alpha_1(D_{\textbf{a}} S_{\mathfrak{i}})}\right)^{-\beta(q)-\varepsilon/2}&\textnormal{by \eqref{Moment Sum Bound}}\\
&\lesssim_{\varepsilon, q} \ \sum_{\mathfrak{i}\in\mathcal{I}_{\delta}}p(\mathfrak{i})^q\alpha_1(D_{\textbf{a}} S_{\mathfrak{i}})^{\beta(q)+\varepsilon/2}\alpha_2(D_{\textbf{a}} S_{\mathfrak{i}})^{\gamma(q)+\varepsilon-\beta(q)-\varepsilon/2}&\textnormal{by \eqref{Alpha Min Bound}}\\
&=\ \ \ \sum_{\mathfrak{i}\in\mathcal{I}_{\delta}}\psi^{\gamma(q)+\varepsilon, q}_{\textbf{a}}(\mathfrak{i})&\textnormal{by \eqref{qmodsv}}\\
&\lesssim_{\varepsilon, q}1.&\textnormal{by Lemma \ref{Sum Lemma}}
\end{align*}
So $\overline{\tau}_{\mu}(q)\leq \gamma(q)+\varepsilon$ by \eqref{lq1}, giving (1) on letting $\varepsilon\to 0$ .

{\it Part (2)}. We suppose $q\geq 1$ and as before let $\delta\in (0, 1]$ and  $\textbf{a}\in[0, 1]^2$. It is sufficient to show that $\underline{\tau}_{\mu}(q)\geq \gamma(q)$. As $q\geq 1$, 
\begin{align*}
\mathcal{D}^q_{\delta}(\mu)=\sum_{Q\in\mathcal{D}_{\delta}}\mu(Q)^q=\sum_{Q\in\mathcal{D}_{\delta}}\Big(\sum_{\mathfrak{i}\in\mathcal{I}_{\delta}}\mu_{\mathfrak{i}}(Q)\Big)^q\geq \sum_{Q\in\mathcal{D}_{\delta}}\sum_{\mathfrak{i}\in\mathcal{I}_{\delta}}\mu_{\mathfrak{i}}(Q)^q&=\sum_{\mathfrak{i}\in\mathcal{I}_{\delta}}\sum_{Q\in\mathcal{D}_{\delta}}\mu_{\mathfrak{i}}(Q)^q
\\ \\
&=\sum_{\mathfrak{i}\in\mathcal{I}_{\delta}}\mathcal{D}^q_{\delta}(\mu_{\mathfrak{i}}).
\end{align*}
Thus for all $\varepsilon>0$,
\begin{align*}
\delta^{\gamma(q)-\varepsilon}\mathcal{D}_{\delta}^q(\mu)\ &\geq\ \delta^{\gamma(q)-\varepsilon}\sum_{\mathfrak{i}\in\mathcal{I}_{\delta}}\mathcal{D}^q_{\delta}(\mu_{\mathfrak{i}})&\\
&\gtrsim\  \delta^{\gamma(q)-\varepsilon}\sum_{\mathfrak{i}\in\mathcal{I}_{\delta}}\mathcal{D}^q_{\widehat{B}_{\mathfrak{i},\delta}}(p(\mathfrak{i})\pi_{\mathfrak{i}}\mu)&\textnormal{by \eqref{hatcomp}}\\
&\gtrsim_{\varepsilon,q}\ \delta^{\gamma(q)-\varepsilon}\sum_{\mathfrak{i}\in\mathcal{I}_{\delta}}p(\mathfrak{i})^q\left(\frac{\widehat{B}\delta}{\alpha_1(D_{\textbf{a}} S_{\mathfrak{i}})}\right)^{-\beta(q)+\varepsilon/2}&\textnormal{by \eqref{Moment Sum Bound}}\\
&\gtrsim_{\varepsilon, q} \ \sum_{\mathfrak{i}\in\mathcal{I}_{\delta}}p(\mathfrak{i})^q\alpha_1(D_{\textbf{a}} S_{\mathfrak{i}})^{\beta(q)-\varepsilon/2}\alpha_2(D_{\textbf{a}} S_{\mathfrak{i}})^{\gamma(q)-\varepsilon-\beta(q)+\varepsilon/2}&\textnormal{by \eqref{Alpha Min Bound}}\\
&=\ \ \ \sum_{\mathfrak{i}\in\mathcal{I}_{\delta}}\psi^{\gamma(q)-\varepsilon, q}_{\textbf{a}}(\mathfrak{i})&\textnormal{by \eqref{qmodsv}}\\
&\gtrsim_{\varepsilon, q} 1.&\textnormal{by Lemma \ref{Sum Lemma}}
\end{align*}
So $\underline{\tau}_{\mu}(q)\geq \gamma(q)-\varepsilon$ giving (2) on letting $\varepsilon\to0$.

{\it Part (3)}
We now assume $\mu$ satisfies the ROSC. Due to Parts (1) and  (2) we only need to provide an upper bound when $q>1$ and a lower bound when $q<1$.

We begin by considering the case when $q>1$. For (1) we obtained an upper bound when $q\in [0, 1]$, but the only place in the proof where we used the assumption $q\leq 1$ was
\[
\mathcal{D}^q_{\delta}(\mu)\leq\sum_{\mathfrak{i}\in\mathcal{I}_{\delta}}\mathcal{D}^q_{\delta}(\mu_{\mathfrak{i}}).
\]
Thus for $q>1$ we shall use the ROSC to show that
\[
\mathcal{D}^q_{\delta}(\mu)\lesssim\sum_{\mathfrak{i}\in\mathcal{I}_{\delta}}\mathcal{D}^q_{\delta}(\mu_{\mathfrak{i}}).
\]
It follows from H\"{o}lder's inequality that for $Q\in\mathcal{D}_{\delta}$
\[
\Big(\sum_{\mathfrak{i}\in\mathcal{I}_{\delta}}\mu_{\mathfrak{i}}(Q)\Big)^q\leq\ k^{q-1}\sum_{\mathfrak{i}\in\mathcal{I}_{\delta}}\mu_{\mathfrak{i}}(Q)^q
\]
where
\begin{equation}\label{k definition}
k:=|\{ \mathfrak{i}\in\mathcal{I}_{\delta}:\mu_{\mathfrak{i}}(Q)>0\}|.
\end{equation}
To complete the proof we need to bound $k$ uniformly for all $\delta$ and $Q\in \mathcal{D}_{\delta}$. Fix $\delta\in (0, 1]$ and $Q\in \mathcal{D}_{\delta}$ such that $\mu(Q)>0$. Let $R$ denote the open unit square $(0, 1)^2$. For convenience if $A>0$ then we write $AQ$ to denote the cube with the same centre as $Q$ but with sidelength $A\delta$.

Let $\mathfrak{i}\in\mathcal{I}_{\delta}$  be such that $S_{\mathfrak{i}}(R)\cap Q$ is non-empty (such an $\mathfrak{i}$ must exist as by assumption $\mu(Q)>0$). Let $\textbf{a}\in S_{\mathfrak{i}}(R)\cap Q$ and consider the vertical ``slice'' of $S_{\mathfrak{i}}(R)$ that contains $\textbf{a}$. By \eqref{Alpha Min Bound} and Lemma \ref{Singular Value Bound}, $g_{\mathfrak{i}, y}(\textbf{a})\asymp\alpha_2(D_{\textbf{a}}S_{\mathfrak{i}})\asymp \delta$. Together with the mean value theorem this implies that the height of this vertical slice is comparable to $\delta$, say it is bounded above by $M\delta$ for some $M>1$ which is independent of $\delta$. 
\begin{figure}[H]
  \centering
\includegraphics[width=7cm]{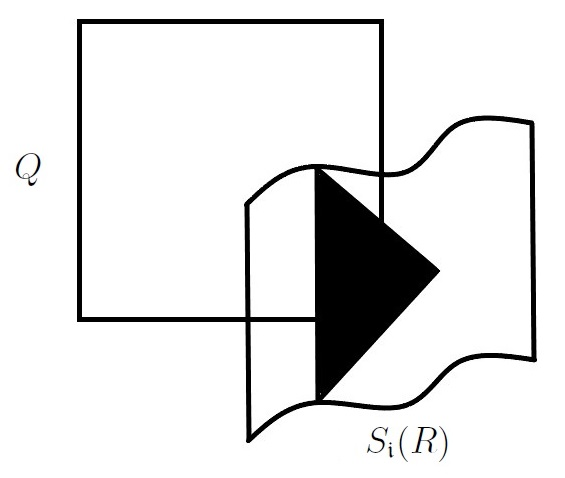}
\caption{$S_{\mathfrak{i}}(R)$ and $Q$, together with the triangle $\Delta_{\mathfrak{i}}$ contained in $S_{\mathfrak{i}}(R)$.} 
\end{figure}

Lemma \ref{Slope Ratio} implies that if we draw a line of slope $L$ (where we can assume $L>1$) from the base of the vertical slice in both directions and a line of slope $-L$ from the top of the vertical slice in both directions then of the two isosceles triangles formed by these lines and the vertical slice at least one must lie within $S_{\mathfrak{i}}(R)$. As the length of the vertical slice is comparable to $\delta$ the area of this triangle is comparable to $\delta^2$. We write $\Delta_{\mathfrak{i}}$ for the triangle which is contained in $S_{\mathfrak{i}}(R)$.

Each triangle $\Delta_{\mathfrak{i}}$ (associated with $\mathfrak{i}\in\mathcal{I}_{\delta}$ such that $S_{\mathfrak{i}}(R)\cap Q \neq \emptyset$) is contained in the square which has the same centre as $Q$ and sidelength $3M\delta$, i.e. the square $3MQ$. Let $\mathcal{L}$ denote two-dimensional Lebesgue measure.
As the area of each $\Delta_{\mathfrak{i}}$ is comparable to $\delta$ and the ROSC guarantees that the interiors of the $\Delta_{\mathfrak{i}}$ are pairwise disjoint,  it follows from \eqref{k definition} that
\begin{equation*}
k\delta^2\lesssim\sum_{\substack{\mathfrak{i}\in\mathcal{I}_{\delta}: \\ S_{\mathfrak{i}}(R)\cap Q \neq \emptyset}}\mathcal{L} (\Delta_{\mathfrak{i}})\leq (3M\delta)^2= 9 M^2 \delta^2.
\end{equation*}
Hence
$k\lesssim 1$ completing the proof of the upper bound for $q>1$.

When $0\leq q<1$  a similar approach to the $q>1$ case above establishes that
\[
\mathcal{D}^q_{\delta}(\mu)\gtrsim\sum_{\mathfrak{i}\in\mathcal{I}_{\delta}}\mathcal{D}^q_{\delta}(\mu_{\mathfrak{i}}).
\]
We omit the proof which is very similar.
\end{proof}

\begin{samepage}

\subsection*{Acknowledgements}

KJF and JMF  were   supported by an \emph{EPSRC Standard Grant} (EP/R015104/1).  
JMF  was  also  supported by   a  \emph{Leverhulme Trust Research Project Grant} (RPG-2019-034). LDL was supported by an \emph{EPSRC Doctoral Training Grant}.

\end{samepage}

\vspace{5mm}
Kenneth J. Falconer, School of Mathematics \& Statistics, University of St Andrews, St Andrews, KY16 9SS, UK
\textit{E-mail address}:\ \url{kjf@st-andrews.ac.uk}

Jonathan M. Fraser, School of Mathematics \& Statistics, University of St Andrews, St Andrews, KY16 9SS, UK
\textit{E-mail address}:\ \url{jmf32@st-andrews.ac.uk}

Lawrence D. Lee, School of Mathematics \& Statistics, University of St Andrews, St Andrews, KY16 9SS, UK
\textit{E-mail address}:\ \url{ldl@st-andrews.ac.uk}

\end{document}